\documentclass[a4paper,9pt]{article}
\usepackage[utf8]{inputenc}

\usepackage{amssymb}
\usepackage{amsmath}
\usepackage{amsthm}
\usepackage{graphicx}
\usepackage[english]{babel}
\usepackage{algorithm}
\usepackage[noend]{algpseudocode}
\usepackage{enumerate}
\usepackage{tikz}
\usepackage{cancel}
\usepackage{mathtools}
\usepackage{hyperref}
\usepackage{csquotes}
\usepackage[backend=biber,style=numeric,sortcites,sorting=nty,natbib,hyperref]{biblatex}
\newtheorem{thm}{Theorem}

\newtheorem{lem}[thm]{Lemma}
\newtheorem{cor}[thm]{Corollary}

\newtheorem{claim}{\small{Claim}}

\newcommand{\R}{
\mathcal{R}
}

\newcommand{\Rep}{
{\mathcal{R}_{\epsilon}}
}

\title{Packing and coloring $r$-bounded axis-parallel rectangles}
\author{Marco Caoduro\vspace{0.4em} \\
\small{Univ. Grenoble Alpes, Laboratoire G-SCOP, Grenoble-INP, Grenoble, France} \\ \small{\textit{\href{mailto:marco.caoduro@grenoble-inp.fr}{marco.caoduro@grenoble-inp.fr}}}}
\date{}

\usepackage{graphicx}

\begin{document}

\maketitle

\begin{abstract}

Let $\R$ be a family of axis-parallel rectangles in the plane. The \emph{transversal number} $\tau(\R)$ is the minimum number of points needed to pierce all the rectangles. The \emph{independence number} $\nu(R)$ is the maximum number of pairwise disjoint rectangles. Given a positive real number $r$, we say that $\R$ is an \emph{r-bounded} family if, for any rectangle in $\R$, the aspect ratio of the longer side over the shorter side is at most $r$.

Gy\'arf\'as and Lehel asked if it is possible to bound the transversal number $\tau(\R)$ with a linear function of the independence number $\nu(\R)$.
Ahlswede and  Karapetyan claimed a positive answer for the particular case of $r$-bounded families, but without providing a proof. Chudnovsky et al. confirmed the result proving the bound $\tau \leq (14 + 2r^2) \nu$. 

This note aims at giving a simple proof of $\tau \leq 2(r+1)(\nu-1) + 1$, slightly improving the previous results.
As a consequence of this new approach, we also deduce a constant factor bound for the ratio $\frac{\chi}{\omega}$ in the case of $r$-bounded family.
\end{abstract}

\section{Introduction}
\subsection{Basic definitions}

Let $\R$ be a family of axis-parallel rectangles in the plane and $r$ a positive real number. We say that $\R$ is an \emph{r-bounded} family if, for any rectangle in $\R$, the aspect ratio of the longer side over the shorter side is at most $r$.

We define some general parameters: 
the \emph{transversal number} $\tau(\R)$ is the minimum number of points needed to pierce all the rectangles in $\R$;
the \emph{independence number} $\nu(\R)$ is the maximum number of pairwise disjoint rectangles in $\R$;
the \emph{chromatic number} $\chi(\R)$ is the minimum number of classes in a partition of $\R$ into pairwise disjoint rectangles,
and the \emph{clique number} $\omega(\R)$ is the maximum number of pairwise intersecting rectangles in $\R$.
It is straightforward to see that $\tau(\R) \geq \nu(\R)$, since if we have a set of pairwise disjoint rectangles we need at least one point for each, and $\chi(\R) \geq \omega(\R)$, since if we have a set of pairwise intersecting rectangles we need a different color for each.

The \emph{intersection graph} $G(\R)$ is the graph having the set of axis-parallel rectangles as its vertex set and an edge for each pair of intersecting rectangles. 
It is easy to see the relations between $\nu(\R)$, $\chi(\R)$, $\omega(\R)$ and usual graph parameters $\alpha(G(\R)), \chi(G(\R)), \omega(G(\R))$ (see \cite{2005_Diestel}). Furthermore, due to the Helly property, $\tau(\R)$ corresponds to the minimum number of cliques covering all the vertices of $G(\R)$.
A family $\R$ is said to be \emph{$k$-degenerate}, if its intersection graph is $k$-degenerate, i.e., the minimum degree of any subgraph of $G(\R)$ is at most $k$.

\subsection{Packing results}
In 1965, Wegner \cite{1965_Wegner} conjectured that, for a family $\R$ of axis-parallel rectangles, we have $\tau(R) \leq 2\nu(R) - 1$. Gy\'arf\'as and Lehel \cite{1985_Gyarfas} asked whether there exists a constant $c$ such that $\tau(R) \leq c\nu(R)$.

Almost sixty years later, both problems are still open. Jel\'inek \cite{2015_Correa} constructed a family of rectangles which shows that $c\geq 2$. Moreover, Chen and Dumitrescu \cite{2020_Chen} and Seb\H{o} \cite{Andras_Note} recently found examples which prove that, if confirmed, Wegner's conjecture is tight for $\nu \leq 4$.

Besides these, there are several partial results in the literature. Kim et al. \cite{2006_Kim} give an answer to Gy\'arf\'as and Lehel's question in general for any set of homothetic images of a fixed convex body with a constant of 16. Moreover, Chan and Har-Peled \cite{2012_Chan} proved the corresponding statement for cross-free families of rectangles (i.e., families of rectangles in which the intersection of two rectangles is either empty or contains at least a corner of one of them).
Complementing this, Asplund and Gr\"unbaum \cite{1960_Asplund} observed that the intersection graph of a crossing family of rectangles $\R_c$ is perfect, so $\nu(\R_c) = \tau(\R_c)$.

Ahlswede and Karapetyan \cite{2006_Ahlswede} published a note with some results without providing the proofs. In \textit{Statement 1}, they claimed that for $r$-bounded families $\tau \leq 2(r+1)\nu$ holds, offering a particular solution of Gy\'arf\'as and Lehel's question. Lately, this statement was confirmed in a paper of Chudnovsky et al. \cite{2018_Chudnovsky} but with a larger factor $(14 + 2r^2)$. In this note we give a simple proof of $\tau \leq 2(r+1)(\nu-1) + 1$, slightly improving the previous results.

\subsection{Coloring results}
Another challenging problem in Geometric Combinatorics is relating the chromatic number and the clique number of family of convex bodies.
In 1948, Bielecki \cite{1948_Bielecki} asked if it is possible to bound the chromatic number of a family of axis-parallel rectangles using a function of the clique number. Asplund and Gr\"unbaum \cite{1960_Asplund} solved the problem providing a quadratic bound which remained asymptotically best for exactly sixty years. Only recently Chamerlsook and Walczak \cite{2020_Chalermsook} improved the result proving $\chi(\R) = \mathcal{O} \left (\omega(\R) \log(\omega(\R)) \right )$.
For a lower bound, Krawczyk and Walczak \cite{2017_Krawczyk} constructed a family $\R$ of axis-parallel rectangles with $\chi(\R) = 3\omega(\R) - 2$  providing the best known lower bound for the problem.

According to the current state of the art and our personal experience, it seems reasonable to ask whether $\frac{\chi}{\omega}$ can be bounded by a constant for families of axis-parallel rectangles.
As an addition to our main theorem, we observe that our proof technique implies a positive answer for $r$-bounded families.

\section{Results}
Let $\R$ be a family of axis-parallel rectangles.
Denote for brevity $\tau = \tau(\R)$ and $\nu = \nu(\R)$.

\begin{thm} \label{Thm_Pac}
	Let $\mathcal{R}$ be a set of axis-parallel rectangles and $r$ be the maximum aspect ratio of a rectangle in $\R$.
	Then,
    $	\tau \leq 2(r+1)(\nu-1) + 1.	$
\end{thm}
\begin{proof}
	We prove the result by induction on $\nu$.
	If $\nu = 1$, then by the Helly property $\tau = 1$, so the inequality holds. \\
	Assume now that $\nu \geq 2$ and the result is true for every family having $\nu' < \nu$.
	Let $\epsilon := \min_{R \in \mathcal{R}} \{ a_R \}$ where $a_R$ is the length of the smaller side of $R$ and let $R_\epsilon$ be the rectangle which realizes this minimum.
	We assume, without loss of generality, that $a_\Rep = \epsilon$ is the height of $R_\epsilon$ and $b_\Rep$ is its width.
	Define $P := \{p_1, ... \ , p_{r+1}, q_1, ... \ , q_{r+1} \}$ as the set of $2(r+1)$ points which subdivide the upper and the lower sides of $R_\epsilon$ into segments of length at most $\frac{b_\Rep}{r} \leq \epsilon$, as showed in Figure \ref{fig::points}.

	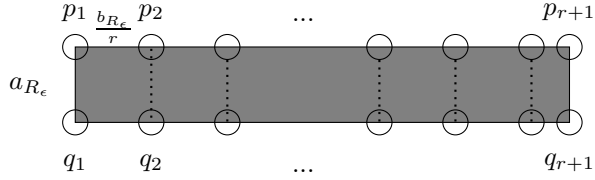
\begin{figure}[h]
	\centering
		\begin{tikzpicture} 

	\filldraw [fill=gray] (0,0) rectangle (6.5,-1);
	
    \node[shape=circle,draw=black] (1) at (0,0) {};
    \node[label={$p_1$}] (1') at (0,0.1) {};
    \node[shape=circle,draw=black] (2) at (1,0) {};
    \node[label={$p_2$}] (1') at (1,0.1) {};
    \draw[dotted, thick] (1') -- (1,-1);
    \node[shape=circle,draw=black] (3) at (2,0) {};
    \draw[dotted, thick] (3) -- (2,-1);
    \node[label={...}] (1') at (3,0.1) {};
    \node[shape=circle,draw=black] (5) at (4,0) {};
    \draw[dotted, thick] (5) -- (4,-1);
    \node[shape=circle,draw=black] (6) at (5,0) {};
    \draw[dotted, thick] (6) -- (5,-1);
    \node[shape=circle,draw=black] (7) at (6,0) {};
    \draw[dotted, thick] (7) -- (6,-1);
    \node[shape=circle,draw=black] (7) at (6.5,0) {};
    \node[label={$p_{r+1}$}] (1') at (6.5,0.1) {};

	\node[shape=circle,draw=black] (1) at (0,-1) {};
	\node[label={$q_1$}] (1') at (0,-1.9) {};
	\node[shape=circle,draw=black] (1) at (1,-1) {};
	\node[label={$q_2$}] (1') at (1,-1.9) {};
	\node[label={...}] (1') at (3,-1.9) {};
	\node[shape=circle,draw=black] (1) at (2,-1) {};
	\node[shape=circle,draw=black] (1) at (4,-1) {};
	\node[shape=circle,draw=black] (1) at (5,-1) {};
	\node[shape=circle,draw=black] (1) at (6,-1) {};
	\node[shape=circle,draw=black] (7') at (6.5,-1) {};
	\node[label={$q_{r+1}$}] (1') at (6.5,-1.9) {};

	\node[label={$a_{R_\epsilon}$}] (0) at (-0.6,-0.9) {};
	\node[label={[font=\footnotesize]$\frac{b_{R_\epsilon}}{r}$}] (0') at (0.5,-0.2) {};

\end{tikzpicture}
		\caption{Definition of $P$}
		\label{fig::points}
	\end{figure}
	
	\begin{claim} \label{R_P}
		If $R \in \R$ and $R \cap R_\epsilon \neq \emptyset$, then $R \cap P \neq \emptyset$.
	\end{claim}
	Indeed, since $R \cap R_\epsilon \neq \emptyset$, then $R$ intersects at least one of the $r$ rectangles with corner points in $P$ partitioning $\R_{\epsilon}$ (Figure \ref{fig::points}).
	We denote this rectangle as $P(R)$: $P(R) \cap R \neq \emptyset$. Both sides of $P(R)$ are of length at most $\epsilon$.
	Since both sides of $R$ are of length at least $\epsilon$, $R$ contains a corner of $P(R)$.
	Hence, $R \cap P \neq \emptyset$ and the claim is proved.
	
	\begin{claim}
		Let $\R' := \{ R \in \R, \ R \cap R_\epsilon = \emptyset \}$, then:
		\begin{enumerate}[a)]
		\item
		$\nu(\mathcal{R}) \geq \nu(\mathcal{R}') + 1$;
		\item
		$\tau(\R) \leq \tau(\R') + 2(r+1)$.
		\end{enumerate}
	\end{claim}
	Indeed, $a)$ follows by noting that adding $R_{\epsilon}$ to any independent set of $\R'$, we get an independent set again;
	$b)$ is an immediate consequence of Claim 1, since $|P| = 2r +1$.

	Claim 2 finishes the proof of Theorem \ref{Thm_Pac} by induction.
\end{proof}

\begin{cor}
	Let $\mathcal{S}$ be a family of squares in the plane, then
	$
		\tau(\mathcal{S}) \leq 4 \nu(\mathcal{S}) - 3.
	$
\end{cor}

In the proof of Theorem \ref{Thm_Pac}, Claim 1 shows that every rectangle intersecting $R_{\epsilon}$ contains a point of $P$. 
Since any point can be contained in at most $\omega(\R)$ elements of $\R$,
$R_{\epsilon}$ has at most $2(r + 1)(\omega(\R)-1)$ neighbors. We can conclude the following:

\begin{lem} \label{LemDeg}
Let $G$ be the intersection graph of an $r$-bounded family of axis-parallel rectangles.
Then, $G$ is $2(r + 1)(\omega(G)-1)$-degenerate.
\end{lem}

\begin{thm} \label{Thm_Col}
    Let $\mathcal{R}$ be a set of axis-parallel rectangles and $r$ be the maximum aspect ratio of a rectangle in $\R$.
	Then,
    $
		\chi(\R) \leq 2(r+1)(\omega(\R)-1) + 1.
	$
\end{thm}

\begin{proof}
    It is enough to use Lemma \ref{LemDeg} and notice that any $k$-degenerate graph has chromatic number at most $k+1$. 
\end{proof}

Theorem \ref{Thm_Col}, together with the following result of Perepelitsa \cite{2003_Perepelitsa}, directly implies another result for squares.

\begin{thm} (Theorem 7, Corollary 8 \cite{2003_Perepelitsa}) \label{PerepThm}
Let G be the triangle-free intersection graph of a finite number of rectangles.
If the intersection of any pair of rectangles is either empty or contains at least a corner of one of them, then G is a plane graph.
Moreover, G is 3-colorable.
\end{thm}

\begin{cor}[Statement 3 \cite{2006_Ahlswede}]
    Let $\mathcal{S}$ be a family of squares in the plane, then
	$
		\chi(\mathcal{S}) \leq 4 \omega(\mathcal{S}) - 3\
	$.
	Moreover, if $G(\mathcal{S})$ is triangle-free, then $\chi(\mathcal{S}) \leq 3$.
\end{cor}

\section{Conclusion}
In this note we offer a simple proof and a slight improvement of a result often cited in the literature but so far without proof. In addition, we state some new results about the chromatic number of $r$-bounded families.

Finally, we would like to point out that a further improvement of the bound proposed by Theorem \ref{Thm_Pac}, either replacing $r+1$ by $r$ or removing the multiplicative factor 2, would lead to an affirmative answer of Wegner's conjecture for families of squares.

\renewcommand{\abstractname}{Acknowledgements}
\begin{abstract}
I would like to thank my thesis supervisor Andr\'as Seb\H{o}, for providing guidance and insightful feedback throughout the writing of this note.
\end{abstract}

\printbibliography[title=References]

\end{document}